\documentclass[12pt,reqno]{amsart}
\numberwithin{equation}{section}
\usepackage[a4paper,margin=1in,footskip=0.25in]{geometry}

\usepackage{tikz}
\usetikzlibrary{calc,through,backgrounds}
\usepackage{multicol}

\usepackage{bm,amssymb,amsthm,amsfonts,amsmath,ytableau,latexsym,cite,color, psfrag,graphicx,ifpdf,pgfkeys,pgfopts,xcolor,extarrows}
\usepackage{graphicx}
\usepackage{subcaption}
\usepackage{amscd,amssymb,amsmath,amsbsy,amsthm}
    \definecolor{linkred}{rgb}{0.7,0.2,0.2}
    \definecolor{linkblue}{rgb}{0,0.2,0.6}
    \definecolor{linkgreen}{rgb}{0,0.6,0.2}
\usepackage[colorlinks,backref,
    linkcolor=linkblue,
    citecolor=linkgreen,
    urlcolor=linkred]{hyperref}

\newtheorem{thm}{Theorem}[section]
\newtheorem{prop}[thm]{Proposition}

\newtheorem{lem}[thm]{Lemma}
\newtheorem{cor}[thm]{Corollary}

\theoremstyle{definition}

\newtheorem{defn}[thm]{Definition}
\newtheorem{remark}[thm]{Remark}

\allowdisplaybreaks

\def\S{\mathfrak{S}}

\def\v{\text{\bf\textit{v}}}
\def\p{\text{\bf\textit{p}}}

\def\e{\text{\bf\textit{e}}}
\def\t{\text{\bf\textit{t}}}

\def\Newton{  \mathrm{Newton}}

\def\supp{  \mathrm{supp}}

\def\S{   \mathfrak{S}  }
\def\G{   \mathfrak{G}  }

\begin{document}

\author{Jinren Dou}
\address{Department of Mathematics, 
Sichuan University, Chengdu, Sichuan 610065, P.R. China}
\email{doujinren@stu.scu.edu.cn, fan@scu.edu.cn, kw\_liu@stu.scu.edu.cn}
\author{Neil J.Y. Fan}
\author{Kunwen Liu}

\title{Lattice-free Schubitopes}
 
\begin{abstract}
In  this paper, we provide a simple criterion  for the Schubitope $\mathcal{S}_{D}$ associated to a diagram $D$  to be  lattice-free. We further show that   $\mathcal{S}_{D}$   is lattice-free if and only if its Ehrhart polynomial is equal to the product of Ehrhart polynomials of the Schubert matroid polytopes corresponding to each column of $D$.  
As applications, we obtain that the Newton polytopes of the Schubert polynomial $\S_w(x)$ and the Grothendieck polynomial $\mathfrak{G}_w(x)$ are lattice-free if and only if 
$w$ avoids the  patterns 1423, 1432, 13254, and confirm several   conjectures by M\'esz\'aros, Setiabrata, and St.\,Dizier on the support of Grothendieck polynomials for this class of permutations. 
\end{abstract}

\maketitle

\section{Introduction}
 Given a $d$-dimensional polytope $P \subset \mathbb{R}^{d}$, let $\mathrm{vert}(P)$ denote its vertex set. If  $\mathrm{vert}(P)\subset \mathbb{Z}^{d}$, then $P$ is called an {\it integral} polytope. If $P\cap \mathbb{Z}^{d}= \mathrm{vert }(P)$, then $P$ is called a \emph{lattice-free} polytope or \emph{empty} polytope \cite{kan,Deza,JMK}. If the interior of 
$P$ contains no lattice points but its boundary might, then   $P$ is called a \emph{hollow} polytope \cite{Zieg}.
Klain \cite{D.A.K}, Manecke and Sanyal \cite{S.Manecke}  studied the enumeration of lattice-free polytopes in the dilation $nP$ of a polytope $P$. Hoffman\cite{AJ}, Arun and Dillon \cite{SA} investigated the vertex enumeration problem for lattice-free polytopes. Furthermore, lattice-free polytopes play special roles in many fields, such as, geometry of numbers\cite{R. K}, integer programming \cite{R. K, HEScarf}, singularities in algebraic geometry \cite{JM,T. Oda}, and so on.

Monical, Tokcan, and Yong  \cite{CNY} initiated the study of  Newton polytopes of many important polynomials in algebraic combinatorics, and introduced the notion of saturated Newton polytopes. For any diagram 
$D$ (a subset of the $n\times n$ grid), they defined a class of generalized permutohedra $\mathcal{S}_{D}$, called   {\it Schubitopes}, 
 and conjectured that the Ehrhart polynomial of 
 $\mathcal{S}_{D}$ has positive coefficients. 
Fink, Mészáros, and St.\,Dizier \cite{FMD} showed that the Schubitope $\mathcal{S}_{D}$ is in fact the Newton polytope of the dual character of the flag Weyl module associated to $D$. In particular, they proved that both the  Newton polytopes of Schubert polynomials and key polynomials are saturated, and are also Schubitopes associated to the Rothe diagram of permutations and skyline diagram of compositions, respectively.  
Recently, Mészáros, Setiabrata, and St.\,Dizier \cite{MSSD} systematically studied the Newton polytopes of Grothendieck polynomials, raising several conjectures regarding their support and saturatedness. 

In  this paper, we  provide a concise criterion for $\mathcal{S}_{D}$ to be a lattice-free polytope, and the Ehrhart polynomial of $\mathcal{S}_{D}$ to be equal to the product of  Ehrhart polynomials of the Schubert matroid polytopes corresponding to each column of $D$.

For a $d$-dimensional  polytope 
$P$,  denote  $i(P, t)$ by the number of lattice points in its  dilated polytope
$tP$, where $t$ is a nonnegative integer, i.e., $i(P, t)=|tP\cap \mathbb{Z}^{d}|.$ By Ehrhart \cite{Ehrhart}, if $P$ is integral, then $i(P,t)$ is a polynomial in $t$ of degree $d$, called the {\it Ehrhart polynomial} of $P$. 
%By convention, the Ehrhart polynomial of the empty set is defined to be 1.

A {\it diagram} $D$ is a collection of  boxes in an $n\times n$ grid, where the boxes are labeled similarly to matrix entries: $(i,j)$ represents the box in the $i$-th row from the top and the $j$-th column from the left. In this paper, we view $D$ as an ordered list of subsets of $[n]:=\{1,2,\ldots,n\}$, i.e., $D=(D_{1}, D_{2},\cdots,D_{n})$, where $D_{j}=\{i:(i,j)\in D\}\subseteq [n]$ is the set of row indices of the boxes in the $j$-th column of $D$. 
The {\it flag Weyl module $\mathcal{M}_{D}$} of the diagram $D$ is a module of the group of invertible upper triangular matrices \cite{WP,WK2,PM}. The {\it dual character of $\mathcal{M}_{D}$}, denoted $\chi_{D}(x)$, is a polynomial in the variables $x_{1},x_{2},\ldots,x_{n}$.

For each column $D_{j}$ of a diagram $D$, 
let $d_j=\max (D_j)$ if $D_j\neq \emptyset$,  and $d_j=0$ otherwise. Define the {\it movable interval} of $D_j$, denoted  $M(D_j)$,  as the interval from the  row index of the  topmost position without a box to $d_j,$ i.e., to the bottommost position with a box in $D_j$. If $D_j=\emptyset$ or $D_j=[d_j]$, then let $M(D_j)=\emptyset.$ For example, in Figure \ref{fig:movable interval}, $M(D_1)=\{2,3\}$, $M(D_2)=\{1,2,3\}$, $M(D_3)=\emptyset$. Here,
\begin{tikzpicture}
    \draw[fill=gray](0mm, 0mm)--(4mm, 0mm)--(4mm, 4mm)--(0mm, 4mm)--(0mm, 0mm);
    \end{tikzpicture}
indicates this position has a box; and \begin{tikzpicture}
    \draw (0mm, 0mm)--(4mm, 0mm)--(4mm, 4mm)--(0mm, 4mm)--(0mm, 0mm);
    \draw (0mm, 4mm)--(4mm, 0mm);
    \draw (0mm, 0mm)--(4mm, 4mm);
    \end{tikzpicture}
indicates this position  has no box.

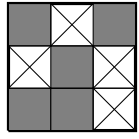
\begin{figure}[hhhhh]
\begin{center}
    \begin{tikzpicture}
    [scale=0.7]
    \draw[very thick](0mm, 0mm)--(24mm, 0mm)--(24mm, 24mm)--(0mm, 24mm)--(0mm, 0mm);
    \draw[fill=gray] (0mm, 0mm)--(8mm,0mm)--(8mm, 8mm)--(0mm, 8mm);
     \draw [fill=gray] (8mm, 0mm)--(16mm,0mm)--(16mm, 8mm)--(8mm, 8mm);
     \draw [fill=gray] (8mm, 8mm)--(16mm,8mm)--(16mm, 16mm)--(8mm, 16mm);
    \draw [fill=gray] (0mm, 16mm)--(8mm, 16mm)--(8mm, 24mm)--(0mm, 24mm);
   \draw [fill=gray] (16mm, 16mm)--(24mm, 16mm)--(24mm, 24mm)--(16mm, 24mm);
   \draw (0mm, 8mm)--(16mm, 24mm);
  \draw (0mm, 16mm)--(8mm, 8mm);
  \draw (8mm, 24mm)--(16mm, 16mm);
  \draw(16mm, 16mm)--(24mm, 8mm);
    \draw (16mm, 8mm)--(24mm, 0mm);
    \draw (24mm, 16mm)--(16mm, 8mm);
   \draw (16mm, 8mm)--(24mm, 8mm)--(16mm, 0mm);
   \draw [step=8mm] (0mm, 0mm) grid (24mm, 24mm);
    \end{tikzpicture}
    \end{center}
\vspace{-2mm}
\caption{$D=(\{1,3\}, \{2 , 3\} ,\{1 \})$}
\label{fig:movable interval}
\end{figure}

\begin{thm}\label{thm: the main thm}
 For any diagram $D=(D_1,\ldots,D_n)\subseteq[n]^2$, the following  conditions are equivalent:
\begin{enumerate}
\item[(1)] The Schubitope $\mathcal{S}_{D}$ is  lattice-free;
\item[(2)] The movable intervals  of different columns of  $D$ intersect in at most one element, i.e., for any $1\le i<j\le n$, $|M(D_i)\cap M(D_j)|\le1$;
\item[(3)] The Ehrhart polynomial of the Schubitope $\mathcal{S}_D$  factors as
\begin{equation}
        i(\mathcal{S}_D, t) =\prod_{j=1}^{n} i(P(SM_{d_j}(D_{j})), t),
\end{equation}
where $P(SM_{d_j}(D_j))$ denotes the Schubert matroid polytope generated by $D_j$.
\end{enumerate}
\end{thm}

Recently,  Li and St.\,Dizier \cite{LST} showed that Schubitopes are not Ehrhart positive in general, thus disproving  \cite[Conjecture 5.19]{CNY}. On the other hand, Ferroni, Morales, and  Panova \cite{FMP} showed that all lattice path matroid polytopes are Ehrhart positive. Since Schubert matroids are  special  lattice path matroids, we see that Schubert matroids  
 are also Ehrhart positive. Thus by Theorem \ref{thm: the main thm}, we obtain that lattice-free Schubitopes are Ehrhart positive.

Let $S_{n}$ denote the symmetric group on $[n]$. For a permutation $w=w(1)w(2)\cdots w(n)\in S_n$, denote its Rothe diagram by $D(w)$.
A permutation $w\in S_n$ is said to contain the permutation pattern $\tau=\tau(1)\tau(2)\dots \tau(k)(k\le n)$ if there exist $1\leq i_1<i_2<\cdots <i_k \leq n$ such that $w(i_1),w(i_2),\ldots,w(i_k)$ have the same relative order as $\tau(1),\tau(2),\ldots,\tau(k)$. If $w$ does not contain   $\tau$, we say $w$  {\it avoids} $\tau$.
A {\it composition} is a sequence of nonnegative integers. For a composition $\alpha$, denote its skyline diagram by $D(\alpha)$.

When the diagram $D$ is the Rothe diagram $D(w)$ of a permutation $w\in S_{n}$, the dual character $\chi_{D}$ of the flag Weyl module associated to $D$  is the Schubert polynomial $\S_{w}(x)$ \cite{WP}; when   $D$ is the skyline diagram $D(\alpha)$ of a composition $\alpha$, $\chi_{D}$ is the key polynomial $\kappa_{\alpha}(x)$ \cite{MD}. For simplicity, we denote  the Newton polytopes of the Schubert polynomial $\S_{w}(x)$ and the
 key polynomial $\kappa_{\alpha}(x)$ as $\mathcal{S}_w$ and $\mathcal{S}_\alpha$, respectively.

\begin{cor}\label{cor: cor 1.2}
For any permutation $w\in S_n$, the following  conditions are equivalent:
\begin{enumerate}
\item[(1)] $\mathcal{S}_w$ is a lattice-free polytope;
\item[(2)] The movable intervals of columns in $D(w)$ are pairwise disjoint;
\item[(3)]  $w$ avoids the   permutation patterns 1423, 1432, 13254;
\item[(4)] The Ehrhart polynomial $i(\mathcal{S}_{w},  t)$  factors as the product of the Ehrhart polynomials $i(P(SM_{d_j}(D_{j})),t)$ of the Schubert matroid polytopes for each column $D_{j}$ of $D(w)$. In particular, $\mathcal{S}_w$ is Ehrhart positive.
\end{enumerate}
\end{cor}

For a permutation $w \in S_n$, the Grothendieck polynomial $\mathfrak{G}_w(x)$ is the $K$-theoretic analog of the Schubert polynomial $\S_w(x)$. It is  well known  that  $\mathfrak{G}_w(x)$ is not homogeneous, and $\S_w(x)$ is the lowest-degree homogeneous component of $\mathfrak{G}_w(x)$. 

\begin{cor}\label{corforGrothendicek}
For any permutation $w\in S_n$, the following  conditions are equivalent:
\begin{enumerate}
\item[(1)] $\Newton(\G_{w})$ is a lattice-free polytope;
\item[(2)] The movable intervals of columns in $D(w)$ are pairwise disjoint;
\item[(3)]  $w$ avoids the   permutation patterns 1423, 1432, 13254;
\item[(4)] The Ehrhart polynomial $i(\Newton(\G_{w}),  t)$  factors as the product of the Ehrhart polynomials $i(P_{\mathrm{sp}}(SM_{d_j}(D_{j})),t)$ of the spanning set polytopes for each column $D_{j}$ of $D(w)$. In particular, $\Newton(\mathfrak{G}_w)$ is Ehrhart positive.
\end{enumerate}
\end{cor}

\begin{cor}\label{cor: cor 1.3}
    For a composition $\alpha$, the following  conditions are equivalent:
\begin{enumerate}
\item[(1)] $\mathcal{S}_{\alpha}$ is a lattice-free polytope;
\item[(2)] $\alpha$ avoids the composition pattern $(0, 2)$, i.e., there do not exist $i<j$ such that $\alpha_{j}-\alpha_{i}\ge 2.$
\item[(3)] The Ehrhart polynomial $i(\mathcal{S}_{\alpha}, t)$  equals the product of the Ehrhart polynomials $i(P(SM_{d_j}(D_{j})),t)$ of the matroid polytopes for each column $D_{j}$ of $D(\alpha)$. In particular, $\mathcal{S}_{\alpha}$ is Ehrhart positive.
\end{enumerate}
\end{cor}

\subsection*{Acknowledgments}
This work is supported by the National Key Research and Development Program of China (No. 2025YFA1017702) and the National Natural Science Foundation of China (No. 12471314). 

\section{Preliminaries}

In this section, we recall the definitions of Schubert polynomials, Grothendieck polynomials, key polynomials, matroid polytopes and Schubitopes.

\subsection{Schubert polynomials and key polynomials}

Schubert polynomials were introduced by Lascoux and Sch\"{u}tzenberger\cite{Lascou} as  polynomial representatives for the cohomology classes of  Schubert varieties in the flag variety.  
Key polynomials, also known as Demazure characters, are polynomials associated to a composition $\alpha$, which  were introduced by Demazure\cite{MD}, and further studied in  \cite{Lascoux2, Lascoux3,Reiner}.

Let $f$ be a polynomial in $x_{1}, \ldots, x_{n}$. The divided difference operator $\partial_{i}$ acting on $f$ is defined as follows:
    $$\partial_{i}f=\frac{f-s_{i}f}{x_{i}-x_{i+1}}, $$
    where $s_{i}f$ denotes the operator that exchanges $x_{i}$ and $x_{i+1}$ in $f$.
 Schubert polynomials $\S_{w}(x)$ and Grothendieck polynomials $\G_{w}(x)$ are defined  recursively as follows. For the longest element $w_{0}=n(n-1)\cdots21 \in S_{n}$, set $\S_{w_{0}}(x)=\G_{w_0}(x)=x_1^{n-1}x_2^{n-2}\cdots x_{n-1}.$
If $w\neq w_0$, then there exists $1 \le i \le n-1$ such that $w(i)<w(i+1)$. Let 
$$\S_{w}(x)=\partial_{i}\S_{ws_{i}}(x), \quad \G_{w}(x)=\partial_{i}(1-x_{i+1})\G_{ws_{i}}(x),$$
where $ws_{i}$ is the permutation obtained by swapping $w(i)$ and $w(i+1)$ in $w$. 

For a composition $\alpha=(\alpha_{1},\alpha_{2},\ldots, \alpha_{n})$,  the  key polynomial $\kappa_{\alpha}(x)$ is defined as below. If $\alpha$ is weakly  decreasing, then set $\kappa_{\alpha}(x)=x_{1}^{\alpha_{1}}x_{2}^{\alpha_{2}}\cdots x_{n}^{\alpha_{n}}$. Otherwise, let $\kappa_{\alpha}(x)=\partial_{i}(x_{i}\kappa_{\alpha'})$, where $\alpha'=(\alpha_{1},\ldots,\alpha_{i+1},\alpha_{i},\ldots, \alpha_{n})$ and $\alpha_{i}<\alpha_{i+1}.$

The {\it Rothe diagram} of a permutation $w\in S_n$ is defined as
\[D(w)=\{(i, j): w(i)>j \ \text{and} \ w^{-1}(j)>i, \forall i, j \in [n]\}.\]
Equivalently, the boxes of 
 $D(w)$ are precisely those in the 
$n\times n$ grid that are not crossed by any rightward or downward ray emanating from the points $(i, w(i))$ for 
$1\leq i\leq n$. The polyline formed by the rightward and downward rays starting from a single point is  called a {\it hook}.
For example, let $w=365142$, then the  boxes of  $D(w)$ are the gray squares shown in Figure \ref{fig:two_plots} (A), and it can be recorded as
$D(365142)=(\{1,2,3\},\{1,2,3,5\},\emptyset,\{2,3\},\{2\},\emptyset).$

For a composition $\alpha=(\alpha_{1},\alpha_{2},\ldots,\alpha_n)$,
the {\it skyline diagram} $D(\alpha)$ of  $\alpha$ is   the subset of the
$n\times n$ grid consisting  of  $\alpha_{i}$  left-justified boxes in the $i$-th row for $i\in[n]$.
Similarly, $D(\alpha)$ can also be represented as $(D(\alpha)_{1},\ldots,D(\alpha)_{n})$, where $D(\alpha)_{j}=\{i\leq n:\alpha_{i}\ge j\}$ for $j\in[n]$.
For example, if $\alpha=(4,1,3,0,2)$, then $D(\alpha)$ is displayed in Figure \ref{fig:two_plots} (B), and $D(\alpha)=(\{1,2,3,5\},\{1,3,5\},\{1,3\},\{1 \},\emptyset).$

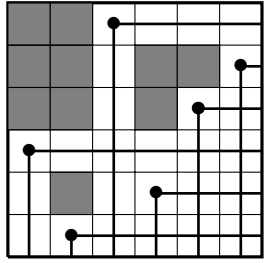
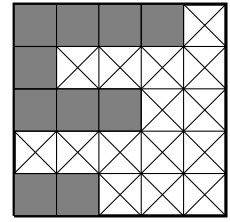
\begin{figure}[htbp]
    \centering
    % 第一个子图
    \begin{subfigure}[b]{0.45\textwidth}
       \centering
        \begin{tikzpicture}[scale=0.7]
            \draw [very thick](0mm, 0mm)--(48mm, 0mm)--(48mm, 48mm)--(0mm, 48mm)--(0mm, 0mm);

    \node at (20mm, 44mm) {$\bullet$};
    \node at (44mm, 36mm) {$\bullet$};
    \node at (36mm, 28mm) {$\bullet$};
    \node at (4mm, 20mm) {$\bullet$};
    \node at (36mm, 28mm) {$\bullet$};
    \node at (28mm, 12mm) {$\bullet$};
    \node at (12mm, 4mm) {$\bullet$};
    \draw [fill=gray] (0mm, 24mm)--(8mm,24mm)--(8mm, 32mm)--(0mm, 32mm);
     \draw [fill=gray] (0mm, 32mm)--(8mm,32mm)--(8mm, 40mm)--(0mm, 40mm);
     \draw [fill=gray] (0mm, 40mm)--(8mm,40mm)--(8mm, 48mm)--(0mm, 48mm);
     \draw [fill=gray] (8mm, 24mm)--(16mm,24mm)--(16mm, 32mm)--(8mm, 32mm);
     \draw [fill=gray] (8mm, 32mm)--(16mm,32mm)--(16mm, 40mm)--(8mm, 40mm);
     \draw [fill=gray] (8mm, 40mm)--(16mm,40mm)--(16mm, 48mm)--(8mm, 48mm);
    \draw [fill=gray] (24mm, 32mm)--(32mm,32mm)--(32mm, 40mm)--(24mm, 40mm);
     \draw [fill=gray] (24mm, 24mm)--(32mm,24mm)--(32mm, 32mm)--(24mm, 32mm);
     \draw [fill=gray] (32mm, 32mm)--(40mm,32mm)--(40mm, 40mm)--(32mm, 40mm);
     \draw [fill=gray] (8mm, 8mm)--(16mm,8mm)--(16mm, 16mm)--(8mm, 16mm);
    \draw [line width=1pt] (20mm, 0mm)--(20mm, 44mm)--(48mm, 44mm);
    \draw [line width=1pt] (4mm, 0mm)--(4mm, 20mm)--(48mm, 20mm);
    \draw [line width=1pt] (28mm, 0mm)--(28mm, 12mm)--(48mm, 12mm);
    \draw [line width=1pt] (36mm, 0mm)--(36mm, 28mm)--(48mm, 28mm);
    \draw [line width=1pt] (44mm, 0mm)--(44mm, 36mm)--(48mm, 36mm);
    \draw [line width=1pt] (12mm, 0mm)--(12mm, 4mm)--(48mm, 4mm);
    \draw [step=8mm] (0mm, 0mm) grid (48mm, 48mm);  
        \end{tikzpicture}
        \subcaption{$D(365142)$}
        \label{sub:D365142}
    \end{subfigure}
    \hfill  % 水平间距
    % 第二个子图
    \begin{subfigure}[b]{0.45\textwidth}
       \centering
        \begin{tikzpicture}[scale=0.7]
            \draw [very thick](0mm, 0mm)--(40mm, 0mm)--(40mm, 40mm)--(0mm, 40mm)--(0mm, 0mm);

            \draw[fill=gray] (0mm, 0mm)--(8mm,0mm)--(8mm, 8mm)--(0mm, 8mm);
            \draw[fill=gray] (8mm, 0mm)--(16mm,0mm)--(16mm, 8mm)--(8mm, 8mm);
            \draw [fill=gray] (0mm, 16mm)--(8mm, 16mm)--(8mm, 24mm)--(0mm, 24mm);
            \draw [fill=gray] (0mm, 24mm)--(8mm, 24mm)--(8mm, 32mm)--(0mm, 32mm);
            \draw [fill=gray] (24mm, 32mm)--(24mm, 40mm)--(32mm, 40mm)--(32mm, 32mm);
            \draw [fill=gray] (8mm, 16mm)--(16mm, 16mm)--(16mm, 24mm)--(8mm, 24mm);
            \draw [fill=gray] (16mm, 16mm)--(24mm, 16mm)--(24mm, 24mm)--(16mm, 24mm); 
            \draw[fill=gray] (0mm, 32mm)--(8mm, 32mm)--(8mm, 40mm)--(0mm, 40mm);
            \draw [fill=gray] (8mm, 32mm)--(16mm, 32mm)--(16mm, 40mm)--(8mm, 40mm);
            \draw [fill=gray] (16mm, 32mm)--(24mm, 32mm)--(24mm, 40mm)--(16mm, 40mm);
            \draw (0mm, 16mm)--(8mm, 8mm);
            \draw (8mm, 16mm)--(24mm, 0mm);
            \draw (16mm, 16mm)--(32mm, 0mm);
            \draw (8mm, 32mm)--(16mm, 24mm);
            \draw (24mm, 16mm)--(40mm, 0mm);
            \draw (16mm, 32mm)--(40mm, 8mm);
            \draw (24mm, 32mm)--(40mm, 16mm);
            \draw (32mm, 32mm)--(40mm, 24mm);
            \draw (32mm, 40mm)--(40mm, 32mm);
            \draw (8mm, 24mm)--(16mm, 32mm);
            \draw (0mm, 8mm)--(8mm, 16mm);
            \draw (16mm, 24mm)--(24mm, 32mm);
            \draw (8mm, 8mm)--(16mm, 16mm);
            \draw (24mm, 24mm)--(40mm, 40mm);
            \draw (16mm, 8mm)--(40mm, 32mm);
            \draw (16mm, 0mm)--(40mm, 24mm);
            \draw (24mm, 0mm)--(40mm, 16mm);
            \draw (32mm, 0mm)--(40mm, 8mm); 
            \draw [step=8mm] (0mm, 0mm) grid (40mm, 40mm);
        \end{tikzpicture}
        \subcaption{$D(4,1,3,0,2)$}
        \label{sub:D32101}
    \end{subfigure}
    \caption{A Rothe diagram and a skyline diagram}
    \label{fig:two_plots}
\end{figure}

\subsection{Schubitopes}

Let $D=(D_1,\ldots,D_n)\subseteq[n]^2$ be an arbitrary diagram. Let $I\subseteq[n]$ be a set of row indices and $j\in[n]$ be a column index. Construct the string $\mathrm{word}_{I}(D_j)$ in the following way. Read the $j$-th column of the $n\times n$ grid from top to bottom and record:
\begin{itemize}
\item[(1)] If $i\notin D_j$ and $i\in I$, record a left parenthesis $($;
\item[(2)]  If $i\in D_j$ and $i\notin I$, record a right  parenthesis $)$;
\item[(3)] If $i\in D_j$ and $i\in I$, record a star $\star$.
\end{itemize}
Let $\theta_{D_j}(I)$ denote the number of matched pairs of parentheses in $\mathrm{word}_{I}(D_j)$ plus the number of $\star$ in $\mathrm{word}_{I}(D_j)$, and  
$\theta_D(I)=\theta
_{D_1}(I)+\cdots+\theta
_{D_n}(I).$

\begin{defn}[\cite{CNY}]
The Schubitope associated to  the diagram $D$ is defined as:
$$\mathcal{S}_D=\left\{\t\in\mathbb{R}_{\geq0}^n\,\Big|\,\sum_{i=1}^nt_i=\#D\text{ and }\sum_{i\in I}t_i\leq\theta_D(I)\text{ for all }I\subseteq[n]\right\}.$$
\end{defn}

A {\it matroid} $M$ is an ordered pair $(E,\mathcal{I})$, where $E$ is a finite set and $\mathcal{I}$ is a family of subsets of $E$,  satisfying: 
(i) $\emptyset \in \mathcal{I}$; 
(ii) If $J\in \mathcal{I}$, and $I\subseteq J$, then $I\in \mathcal{I}$; 
(iii) If $I,J\in \mathcal{I}$, and $|I|<|J|$, then there exists $j\in J\backslash I$ such that $I\cup \{j\}\in \mathcal{I}$.  
 By (iii), all bases of a matroid have the same size. Therefore, we often denote a matroid as $M=(E,\mathcal{B})$, where $\mathcal{B}$ is the set of bases of $M$.

For a subset $S\subseteq[n]$, the {\it Schubert matroid} $SM_{n}(S)$ is the matroid with ground set $[n]$, and bases $\{T\subseteq[n]:T\preceq S\}$, where $T\preceq S$ means that $T$ is less than or equal to $S$ in the Gale order, i.e., 
(i) $|T|=|S|$; (ii) If we write $T=\{a_{1}<a_{2}<\cdots< a_{k}\}$ and $S=\{b_{1}<b_{2}<\cdots< b_{k}\}$, then  $a_{i}\leq b_{i}$ for $1\le i\le k$.

Given a matroid $M=([n],\mathcal{B})$, the associated {\it matroid base polytope} is constructed as follows. Let $\{ \e_{i}:1\leq i\leq n\} $ be the standard basis vectors of $\mathbb{R}^{n}$. For a subset $B=\{b_{1},\ldots,b_{k}\} \subseteq [n]$, denote $\e_{B}=\e_{b_{1}}+\cdots+\e_{b_{k}}$ by the indicator vector of the set $B$.
The matroid base polytope $P(M)$ of $M$ is defined as  the convex hull 
$$P(M)=\text{conv}\{\e_{B}:B\in \mathcal{B}\},$$
which can also  be defined by the rank function $r:2^{[n]}\rightarrow \mathbb{Z}_{\ge0}$ of $M$: 
$$P(M)=\left\{\t\in \mathbb{R}^n\,\Big|\, \sum_{i\in I}t_i\leq r(I)\text{ for all }I\subseteq [n],\text{ and }\sum_{i\in E}t_i=r([n])\right\}.$$

Let $\mathcal{S}$ denote the collection of spanning sets of $M$, that is, the subsets of $[n]$ containing a basis of $M$. The \emph{spanning set polytope} $P_{\mathrm{sp}}(M)$ is defined as the convex hull
\[
P_{\mathrm{sp}}(M) = \operatorname{conv}\{\e_S \mid S \in \mathcal{S}\}.
\]
The spanning set polytope can also be parameterized by
$$P_{\mathrm{sp}}(M)=\left\{\t\in \mathbb{R}^n\,\Big|\, r([n])-r([n]\backslash I)\leq \sum_{i\in I}t_i\leq |I|, \text{ for all  }I\subseteq [n]\right\}.$$
Let \(\mathcal{I}\) denote the
collection of independent sets of \(M\), that is, the subsets of \([n]\)
contained in a basis of \(M\). The \emph{independence polytope}
\(P_{\mathrm{ind}}(M)\) is defined as the convex hull
\[
P_{\mathrm{ind}}(M)
=
\operatorname{conv}
\left\{
\e_{I}\mid I\in\mathcal{I}
\right\}.
\]
Equivalently, in terms of the rank function of \(M\), the independence
polytope is given by
\[
P_{\mathrm{ind}}(M)
=
\left\{
\t\in\mathbb{R}_{\geq 0}^{n}
\;\Big|\;
\sum_{i\in I}t_i\leq r(I),
\text{ for all }I\subseteq [n]
\right\}.
\]
In particular, the inequalities corresponding to the singleton subsets
imply \(0\leq t_i\leq r(\{i\})\leq 1\) for every \(i\in E\). Moreover, the
matroid base polytope is the top face of the independence polytope.

A function $z:2^{[n]}\rightarrow \mathbb{R}$ is called \emph{submodular} if
$z(I)+z(J)\geq z(I\cap J)+z(I\cup J)$ for all $I,J\subseteq [n]$ and called \emph{supermodular} if $-z$ is submodular. 
A \emph{paramodular pair} is a pair $(y, z)$ of functions $2^{[n]} \to \mathbb{R}$ such that $y(\emptyset) = z(\emptyset) = 0$, $y$ is supermodular, $z$ is submodular, and $z(I) - y(J) \ge z(I \setminus J) - y(J \setminus I)$ for all $I, J \subseteq [n]$. A \emph{generalized polymatroid} is a polytope $Q \subset \mathbb{R}^n$ for which there exists a paramodular pair $(y, z)$ such that
\[
Q = \Big\{\t \in \mathbb{R}^n \;\Big|\; y(I) \le \sum_{i \in I} t_i \le z(I), \text{ for all } I \subseteq [n]\Big\}.
\]

Let $\rho : 2^{[n]} \to \mathbb{Z}_{\ge 0}$ be a monotone submodular function with $\rho(\emptyset) = 0$. An \emph{integral polymatroid}  is defined as
\[
    P(\rho) = \Big\{\t \in \mathbb{R}_{\ge 0}^n \;\Big|\; \sum_{i \in S} t_i \le \rho(S), \text{ for all } S \subseteq [n]\Big\}.
\]
This is a special case of a generalized polymatroid. Another important special case is the \emph{generalized permutohedron}, which is a polytope defined by
\[
P = \left\{\t \in \mathbb{R}^n \;\Big|\; \sum_{i \in I} t_i \le z(I), \text{ for  } I \subseteq [n] \text{ and } \sum_{i=1}^n t_i = z([n])\right\}
\]
for some submodular function $z : 2^{[n]} \to \mathbb{R}$ satisfying $z(\emptyset) = 0$. Specifically, the matroid base polytope $P(M)$ is a generalized permutohedron determined by the rank function $r(M)$ of $M$; see \cite{A.S}. Finally, for two polytopes $P$ and $Q$, their Minkowski sum is defined as $P + Q = \{u + v \mid u \in P, v \in Q\}$.

Denote $\chi_{D}$  by the {\it dual character} of the flag Weyl module  associated to any diagram $D$, see 
\cite{NG3, FMD} for precise definition. Recall that the {\it Newton polytope} $\Newton(f)$ of a polynomial $f=\sum_{\alpha}c_{\alpha}x^\alpha$ is the convex hull of exponent vectors of $f$, namely,
$\Newton(f)=\text{conv}\{\alpha: c_\alpha\neq 0\}$.
It turns out that both the Newton polytope of $\chi_{D}$ and the Schubitope of $D$ are  the Minkowski sum of  Schubert matroid polytopes corresponding to each column of $D$.
\begin{thm}[\cite{FMD}]\label{S_D and minkovski sum}
Let $ D = (D_1, \ldots, D_n)$ be a diagram.  Then
\begin{align*}
\mathcal{S}_D =\Newton(\chi_{D})= P(SM_{d_1}(D_1)) + \cdots + P(SM_{d_n}(D_n)).
\end{align*}
\end{thm}

\section{Lattice-free Schubitopes}

Recall that the Schubert matroid $SM_{d_j}(D_{j})$  for some $j\in[n]$ is the matroid with ground set $[n]$ and bases $\{T\subseteq[d_j]:T\preceq D_{j}\}$. For a subset $T$  smaller than $D_{j}$ in the Gale order, we can intuitively view it as the set of row indices of boxes in column $D_j$ after pushing the movable boxes upwards to available positions  (i.e., positions not occupied by other boxes).

In a given column $D_{j}$ of a diagram $D$, a box at position $(i,j)$ is called {\it movable} if there exists a position $(i', j)$ above it (with $i'<i$) that has no box; otherwise, it is called {\it immovable}. The set of row indices corresponding to   all positions that the movable boxes in this column can move to (including the row where the box itself is located) is called the {\it movable interval} of $D_{j}$. Clearly, this definition is consistent with the definition of movable interval given in Introduction, and $M(D_j)$ is an interval, i.e., a set of consecutive positive integers. Let $[a, b]=\{a, a+1, \ldots , b-1, b\}$ be the interval from $a$ to $b$. 

We use boldface letters to denote vectors. 

\begin{thm}\label{thm: main.lattice-free}
For any diagram $D=(D_1,\ldots,D_n)$, the Schubitope $\mathcal{S}_{D}$ is lattice-free  if and only if $|M(D_i)\cap M(D_j)|\le 1$ for any $1\le i<j\le n$.
\end{thm}

\begin{proof}
We first prove the necessity. By contradiction, we  show that if there are two columns of $D$, say $D_1,D_2$,  such that $|M(D_1)\cap M(D_2)|\ge 2$, then $\mathcal{S}_{D}$ must contain  an non-vertex lattice point.  Assume that there are two integers  $p<q$ such that $\{p,q\}\subseteq M(D_1)\cap M(D_2)$. Then there exist $T_{1},T_{2}\preceq D_{1}$ and $S_{1},S_{2}\preceq D_{2}$, such that $\e_{T_{1}}=(a_{1},\ldots,1,\ldots,0,\ldots , a_{n})$ and $\e_{T_{2}}=(a_{1},\ldots,0,\ldots,1,\ldots ,a_{n})$ have the same values in all components except the $p$-th and $q$-th components; similarly, $\e_{S_{1}}=(b_{1},\ldots,1,\ldots,0,\ldots ,b_{n})$ and $\e_{S_{2}}=(b_{1},\ldots,0,\ldots,1,\ldots ,b_{n})$ have the same values in all components except the $p$-th and $q$-th components.

For $i\in [3, n]$, fix $L_{i}\preceq D_{i}$, and let $\e_{L}=\sum_{i=3}^{n}\e_{L_{i}}=(l_1, \ldots , l_n)$. Then the point $\v_{1}=\e_{T_{1}}+\e_{S_{1}}+\e_{L}=(a_{1}+b_{1}+l_1,\ldots,2+l_p,\ldots,l_q,\ldots, a_{n}+b_{n}+l_n)$ and the point $\v_{2}=\e_{T_{2}}+\e_{S_{2}}+\e_{L}=(a_{1}+b_{1}+l_1,\ldots,l_p,\ldots,2+l_q,\ldots, a_{n}+b_{n}+l_n)$ are two distinct points in $\mathcal{S}_{D}$, and the midpoint $\overline{\v}$ of $\v_{1}$ and $\v_{2}$ is 
\begin{align}\label{mdp}
\overline{\v}=(a_{1}+b_{1}+l_1,\ldots,1+l_p,\ldots,1+l_q,\ldots, a_{n}+b_{n}+l_n),
\end{align}
which is a non-vertex lattice point of $\mathcal{S}_{D}$. Moreover, $\overline{\v}$ can be obtained by at least two ways, i.e., $\overline{\v}=\e_{T_1}+\e_{S_2}+\e_L=\e_{T_2}+\e_{S_1}+\e_L$.

Now we prove the sufficiency. After deleting all the empty columns and up-justified columns in $D$. There are two cases. 

\noindent {\bf Case 1}. The movable intervals of boxes from different columns in $D$ are pairwise disjoint, i.e., for any distinct $i,j\in [n]$, $M(D_{i})\cap M(D_{j})=\emptyset$. Consequently, the nonconstant coordinates of $P(SM_{d_i}(D_{i}))$ and $P(SM_{d_j}(D_{j}))$ lie in $M(D_{i})$ and $M(D_{j})$ respectively, which are disjoint. Since $P(SM_{d_i}(D_{i}))$ and $P(SM_{d_j}(D_{j}))$ are 0/1 polytopes,  it is easy to see that their Minkowski sum is lattice-free.

\noindent {\bf Case 2}. There exist at least two columns in  $D$ whose movable intervals intersect in exactly one element, i.e., $\exists  i,j\in [n]$, such that $|M(D_{i})\cap M(D_{j})|= 1$. Partition the columns of $D$ into different blocks, such that the movable intervals of   columns from different blocks do not intersect. That is, let $D=\{D_{1}, D_{2},\ldots , D_{n} \}=\bigcup_{i=1}^{d} \mathcal{D}_{i}$, where for $ 1\leq i<j \leq d$,   columns in $\mathcal{D}_{i}$ and  columns in $\mathcal{D}_{j}$ have no intersecting movable intervals. Therefore,  it suffices to consider the Minkowski sum of Schubert matroid polytopes from a single block. Without loss of generality, let us consider the block $\mathcal{D}=\{D_{1}, D_{2}, \ldots ,D_{k}\}$.

For $1\le j\le k$, let $m_{j}=\min M(D_{j})$, and define $m_{k+1}=\max M(D_{k})$.
Since  each $M(D_{i})$ is a set of consecutive positive integers containing at least 2 elements, and the movable intervals in different columns of $\mathcal{D}$ intersect in at most one element, we find that  $m_{i}\neq m_{j}$ for $i\neq j$.  Without loss of generality, we can assume that the columns of $\mathcal{D}$ are arranged such that $m_1<m_2<\cdots<m_k$, as shown in Figure \ref{fig:n}.
\begin{figure}[hhhhh]
 \begin{center}
    \begin{tikzpicture}
[scale=0.8]
    \draw [very thick](0mm, 0mm)--(64mm, 0mm)--(64mm, 88mm)--(0mm, 88mm)--(0mm, 0mm);
    \draw [fill=gray] (40mm, 8mm)--(48mm,8mm)--(48mm, 16mm)--(40mm, 16mm)--(40mm, 8mm);
    \draw [fill=gray] (32mm, 32mm)--(48mm, 32mm)--(48mm, 88mm)--(16mm, 88mm)--(16mm, 64mm)--(24mm, 64mm)--(24mm, 48mm)--(32mm, 48mm);
    \draw  (40mm, 24mm)--(48mm, 24mm)--(48mm, 32mm)--(40mm, 32mm)--(40mm, 24mm)--(48mm, 32mm);
    \draw  (48mm, 24mm)--(40mm, 32mm);
    \node[above] at (44mm, 88mm){$D_{k}$};
    \node[above] at (36mm, 88mm){$\cdots$};
    \node[below] at (44mm, 26mm){$\vdots$};
     \node[left] at (40mm, 28mm){$\cdots$};
    \draw (24mm, 40mm)--(32mm, 40mm)--(32mm, 48mm)--(24mm, 48mm)--(24mm, 40mm)--(32mm, 48mm);
    \draw (32mm, 40mm)--(24mm, 48mm);
    \node[above] at (28mm, 88mm){$D_{3}$};
    \node[below] at (28mm, 44mm){$\vdots$};
     \draw [fill=gray] (16mm, 40mm)--(24mm, 40mm)--(24mm, 48mm)--(16mm, 48mm)--(16mm, 40mm);
    \draw (16mm, 56mm)--(24mm, 56mm)--(24mm, 64mm)--(16mm, 64mm)--(16mm, 56mm)--(24mm, 64mm);
     \draw (24mm, 56mm)--(16mm, 64mm);
    \node[above] at (20mm, 88mm){$D_{2}$};
      \node[below] at (20mm, 58mm){$\vdots$};
   \draw [fill=gray] (8mm, 56mm)--(16mm, 56mm)--(16mm, 64mm)--(8mm, 64mm)--(8mm, 56mm);
    \draw  (8mm, 48mm)--(16mm, 48mm)--(16mm, 56mm)--(8mm, 56mm)--(8mm, 48mm);
    \draw  (8mm, 48mm)--(16mm, 56mm);
   \draw  (8mm, 56mm)--(16mm, 48mm);
   \draw  (8mm, 40mm)--(16mm, 40mm)--(16mm, 48mm)--(8mm, 48mm)--(8mm, 40mm);
    \draw  (8mm, 40mm)--(16mm, 48mm);
   \draw  (8mm, 48mm)--(16mm, 40mm);
    \draw  (8mm, 32mm)--(16mm, 32mm)--(16mm, 40mm)--(8mm, 40mm)--(8mm, 32mm);
    \draw  (8mm, 32mm)--(16mm, 40mm);
   \draw  (8mm, 40mm)--(16mm, 32mm);
\node[below] at (12mm, 34mm){$\vdots$};
\node[below] at (12mm, 26mm){$\vdots$};
\node[below] at (12mm, 18mm){$\vdots$};
  \draw  (16mm, 32mm)--(24mm, 32mm)--(24mm, 40mm)--(16mm, 40mm)--(16mm, 32mm);
    \draw  (16mm, 32mm)--(24mm, 40mm);
   \draw  (16mm, 40mm)--(24mm, 32mm);
\node[below] at (20mm, 34mm){$\vdots$};
\node[below] at (20mm, 26mm){$\vdots$};
\node[below] at (20mm, 18mm){$\vdots$};
     \draw (8mm, 72mm)--(16mm, 72mm)--(16mm, 80mm)--(8mm, 80mm)--(8mm, 72mm)--(16mm, 80mm);
      \draw (16mm, 72mm)--(8mm, 80mm);
    \draw  (8mm, 0mm)--(16mm, 0mm)--(16mm, 8mm)--(8mm, 8mm)--(8mm, 0mm)--(16mm, 8mm);
    \draw  (16mm, 0mm)--(8mm, 8mm);
    \draw  (16mm, 0mm)--(24mm, 0mm)--(24mm, 8mm)--(16mm, 8mm)--(16mm, 0mm)--(24mm, 8mm);
    \draw  (24mm, 0mm)--(16mm, 8mm);
    \draw  (24mm, 0mm)--(32mm, 0mm)--(32mm, 8mm)--(24mm, 8mm)--(24mm, 0mm)--(32mm, 8mm);
    \draw  (32mm, 0mm)--(24mm, 8mm);
    \draw  (32mm, 0mm)--(40mm, 0mm)--(40mm, 8mm)--(32mm, 8mm)--(32mm, 0mm)--(40mm, 8mm);
    \draw  (40mm, 0mm)--(32mm, 8mm);
      \draw  (40mm, 0mm)--(48mm, 0mm)--(48mm, 8mm)--(40mm, 8mm)--(40mm, 0mm)--(48mm, 8mm);
    \draw  (48mm, 0mm)--(40mm, 8mm);
       \node[left] at (0mm, 76mm){$m_{1}$};
      \node[left] at (0mm, 60mm){$m_2$};
      \node[left] at (0mm, 44mm){$m_3$};
      \node[left] at (0mm, 28mm){$m_k$};
      \node[left] at (0mm, 12mm){$m_{k+1}$};
     \node[below] at (12mm, 74mm){$\vdots$};
      \draw [fill=gray] (8mm,80mm)--(16mm, 80mm)--(16mm, 88mm)--(8mm, 88mm)--(8mm, 80mm);
      \node[above] at (12mm, 88mm){$D_{1}$};
    \end{tikzpicture}
    \end{center}
\vspace{-4mm}
\caption{$\mathcal{D}=\{D_{1}, D_{2}, \ldots ,D_{k}\}$}
\label{fig:n}
\end{figure}
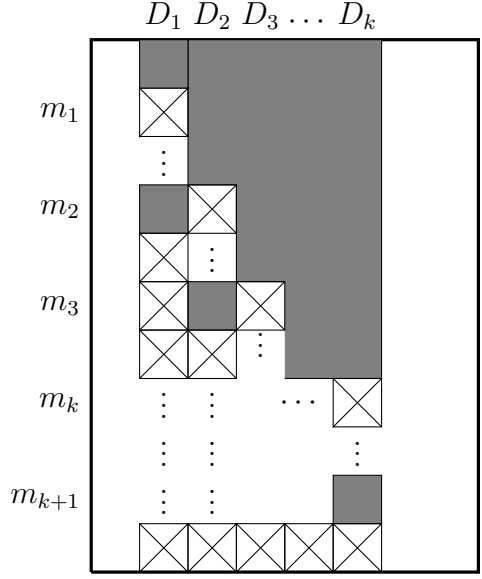

Suppose that $\v_{1},\v_{2}, \ldots, \v_{l}$ are any $l$ distinct vertices in $\sum _{j=1}^{k}P(SM_{d_j}(D_{j}))$. If the convex hull of $\v_{1}, \v_{2},\ldots,\v_{l}$ contains a lattice point $\v'$, and $\v' \neq \v_{1},\ldots,\v_{l}$, then  $\v'$ can be written as 
\[\v'=\lambda_{1}\v_{1}+\lambda_{2}\v_{2}+\cdots +\lambda_{l}\v_{l},\] 
where $0< \lambda_{i}< 1$ and $\lambda_{1}+\lambda_{2}+\cdots+\lambda_{l}=1$. 
We will show that $\v'=\v_1=\v_2=\cdots=\v_l$, leading to a contradiction.

By our assumption of the set $\mathcal{D}$, we have: (i) For any $j\in[k]$, the lowest movable box in $D_{j}$ is in row $m_{j+1}$, i.e., $M(D_{j})=[m_{j}, m_{j+1}]$. (ii) Except for $D_{1}$, the movable intervals  in other columns do not contain $[m_{1}, m_{2}-1]$, meaning the other columns have boxes in every row from $[m_{1}, m_{2}-1]$. Therefore, except for $D_{1}$, the basis of the Schubert matroid for every other column contains the interval $[m_{1}, m_{2}-1]$.
(iii) For $ i < m_1 $, the boxes in row $i$ of $D_{1}, \ldots, D_{k}$ are all immovable; for $ i> m_{k+1}$ , there are no boxes in row $i$ of $D_{1}, \ldots, D_{k}$. Hence, for any $i<m_1$ or $i> m_{k+1}$, the $i$-th components of $\v_{1},\v_{2}, \ldots, \v_{l}$ are all equal.

Now consider the $m_{1}$-th component of $\v'$. Notice that for $1\le j\le l$, 
\begin{align}\label{vj}
\v_j=\e_{B^{1}_{j}}+\cdots+\e_{B^{k}_{j}},
\end{align}
where $B^{j}_{r}$ is a basis of $SM_{d_r}(D_{r})$ for $r\in[k]$. Let $v_{j,i}$ denote the $i$-th component of $\v_{j}$. Then the $i$-th component of $\v'$ is
\begin{equation}\label{eq:the i-th component}
v'_i=\lambda_{1}v_{1,i}+\lambda_{2}v_{2,i}+\cdots +\lambda_{l}v_{l,i}.
\end{equation}
By (ii), for any $1\le j\le l$, $v_{j,m_{1}}=x_{j,m_{1}}+k-1$, where $x_{j,m_{1}}\in \{0,1\}$. By equation \eqref{eq:the i-th component}, we obtain
\begin{align*}       v'_{m_1}&=\lambda_{1}v_{1,m_{1}}+\lambda_{2}v_{2,m_{1}}+\cdots +\lambda_{l}v_{l,m_{1}}  \notag\\
&=\lambda_{1}x_{1,m_{1}}+\lambda_{2}x_{2,m_{1}}+\cdots +\lambda_{l}x_{l,m_{1}} +(\lambda_1+\cdots +\lambda_l) (k-1) \notag\\
&=\lambda_{1}x_{1,m_{1}}+\lambda_{2}x_{2,m_{1}}+\cdots +\lambda_{l}x_{l,m_{1}}+(k-1).
\end{align*}
Therefore, $\lambda_{1}x_{1,m_{1}}+\lambda_{2}x_{2,m_{1}}+\cdots +\lambda_{l}x_{l,m_{1}}=v'_{m_1}-(k-1)$ is an integer. Since $\lambda_{1}+\lambda_{2}+\cdots+\lambda_{l}=1$, and $0< \lambda_{i}< 1$, it follows that $x_{1,m_{1}}=x_{2,m_{1}}=\cdots=x_{l,m_{1}}=0\text{ or }1$, implying the $m_{1}$-th components of $\v_{1}, \ldots, \v_{l}$ are the same. Using the same argument for the $m_{1}+1,  \ldots, (m_{2}-1)$-th components, we can conclude that the $m_{1}+1,  \ldots, (m_{2}-1)$-th components of $\v_{1}, \ldots, \v_{l}$ are all equal.

Now consider the row $m_2$.
Since the bases of a matroid have the same size, and the $m_{1}, m_{1}+1,  \ldots, (m_{2}-1)$-th components of $\v_{1}, \ldots, \v_{l}$ are all the same, we find that the $m_2$-th component of $\e_{B_{j_1}}$ is the same  for each $\v_j$ in equation \eqref{vj}.
That is to say, the contribution of $D_1$ to each of $\v_1,\ldots,\v_l$ are the same.  Therefore, we can apply the same argument to the $m_{2}$-th component, and conclude that the $m_{2}$-th components of $\v_{1}, \ldots, \v_{l}$ are equal. Continuing this process, we can derive that for $ i\in [m_{1}, m_{k+1}]$, the $i$-th components of $\v_{1}, \ldots, \v_{l}$ are all equal. Hence $\v_{1}, \ldots, \v_{l}$ are the same point, a contradiction.
\end{proof}

\section{Ehrhart Polynomial of lattice-free Schubitopes} 
In this section, we factor the Ehrhart polynomial of lattice-free Schubitopes.

\begin{lem}[{\cite[Theorem 44.6, Corollary 46.2c]{A.S}}]
\label{minkowski_of_ gp}
Let $P_1,\ldots,P_k$ be integral polymatroids. Then the Minkowski sum $P_1 + \cdots + P_k$  is an integral polymatroid and has  the integer decomposition property (IDP). That is, 
\[
\left( \sum_{j=1}^k P_j \right) \cap \mathbb{Z}^n = \sum_{j=1}^k (P_j \cap \mathbb{Z}^n).
\]
\end{lem}

\begin{remark}\label{rmk:IDP}
By \cite[Corollary~42.1e]{A.S}, the base polytope
$P(M)$ of every matroid $M$ has the integer decomposition
property. Note that the base polytope is the maximal face of the independence polytope (which is a polymatroid) for any matroid. 
For matroids $M_1,\ldots,M_k$, any integer point $x \in \sum P(M_j)$ also belongs to the sum of the corresponding polymatroids $\sum P_{\mathrm{ind}}(M_j)$. Lemma \ref{minkowski_of_ gp}, $x$ can be decomposed as $x = \sum x_j$ with $x_j \in P_{\mathrm{ind}}(M_j) \cap \mathbb{Z}^n$. Since $x\in \sum P(M_j)$, the sum of the coordinates of $x$ strictly equals $\sum r(M_j)$, and each $x_j$ satisfies the polymatroid bound $\sum (x_j)_i \le r(M_j)$, it is forced that $\sum (x_j)_i = r(M_j)$ for every $j$. Thus, every $x_j$  lies in the base polytope $P(M_j)$, implying the mixed integer decomposition property: 
\[
\left(\sum_{j=1}^{k} tP(M_j)\right)\cap\mathbb{Z}^{n}
=
\sum_{j=1}^{k}
\left(tP(M_j)\cap\mathbb{Z}^{n}\right).
\]

\end{remark} 
 
\begin{thm}\label{thm: main.Ehrhart}
    For any diagram $D=(D_1,\ldots,D_n)$, $i(\mathcal{S}_{D},t)=\prod_{j=1}^ni(P(SM_{d_j}(D_{j})),t)$ if and only if the movable intervals  of two different columns of $D$ intersect in at most one element.
\end{thm}
\begin{proof}
We first prove the necessity. Suppose to the contrary that there are two columns in  $D$ whose movable intervals intersect in at least two elements. Without loss of generality, assume that the two columns $D_{1}$ and $D_{2}$ satisfy  $|M(D_{1}) \cap M(D_{2})|\ge 2$. By the proof of Theorem \ref{thm: main.lattice-free}, there is a non-vertex lattice point $\overline{\v}$ as in \eqref{mdp} can be obtained by at least two ways.
Hence, 
\[|\mathcal{S}_{D} \cap \mathbb{Z}^n| = \left| \left( \sum_{j=1}^{n} P(SM_{d_j}(D_{j})) \right) \cap \mathbb{Z}^n \right| < \prod_{j=1}^n |P(SM_{d_j}(D_{j})) \cap \mathbb{Z}^n|,\]
which is a contradiction.

Now we  prove the sufficiency.  Since every dilation of  matroid polytope $tP(M)$ is a  generalized permutohedron,  by Lemma \ref{minkowski_of_ gp} and Remark \ref{rmk:IDP}, we have  
    \begin{align}\label{eq3.1}
    t\mathcal{S}_{D}\cap \mathbb{Z}^{n}   &=\bigg( \sum _{j=1}^{n} tP(SM_{d_j}(D_{j}))\bigg)\cap\mathbb{Z}^{n}\nonumber\\
    &=\sum _{j=1}^{n}\bigg(tP(SM_{d_j}(D_{j}))\cap\mathbb{Z}^{n}\bigg).
    \end{align}   
It suffices to show that every lattice point in $t\mathcal{S}_{D}\cap \mathbb{Z}^{n}$ is uniquely determined in \eqref{eq3.1}, or equivalently,
    \begin{equation}\label{eq3.2}
    \left | t\mathcal{S}_{D}\cap \mathbb{Z}^{n}\right|=\prod_{j=1}^n\left|tP(SM_{d_j}(D_{j}))\cap\mathbb{Z}^{n}\right| .
    \end{equation}
    
By \eqref{eq3.1}, for any $\v \in t\mathcal{S}_{D}\cap \mathbb{Z}^{n}$, we can write 
\begin{equation}\label{eq5}
\v=\p_{1}+\p_{2}+\cdots+\p_{n} ,
\end{equation}
where $\p_j \in tP(SM_{d_j}(D_{j}))\cap \mathbb{Z}^{n}.$
For $1\le j \le n$, suppose that $SM_{d_j}(D_{j})$ has $s_{j}$ bases. Let ${B^{j}_{1}}, {B^{j}_{2}},\cdots,{B^{j}_{s_{j}}}$ be all the bases of $SM_{d_j}(D_{j})$, then $\{\e_{B^{j}_{i}}, i\in [s_{j}]\}$ are  all the vertices of $P(SM_{d_j}(D_{j}))$ and $\{t\e_{B^{j}_{i}}, i\in[s_{j}]\}$ are  all the vertices of $tP(SM_{d_j}(D_{j}))$, thus we have  
\begin{equation}\label{pjj}
  \p_{j}=  \sum _{i=1}^{s_{j}}\lambda_{i}t \e_{B^{j}_{i}}.  
\end{equation}
     where $0 \leq \lambda_{i}\leq 1$ for $i\in [s_{j}]$, and $\lambda_{1}+\lambda_{2}+\cdots+\lambda_{s_{j}}=1.$
 
There are two cases to consider.

\noindent {\bf Case 1}. The movable intervals of  columns in $D$ are pairwise disjoint.  Consider the $d$-th component $v_d$ of $\v=(v_1,\ldots,v_n)$ in \eqref{eq5}.  For any $1\le j\le n$, if $d<\min(M(D_j))$, then there is an immovable box in row $d$ of column $D_{j}$. Recall that ${B^{j}_{i}}$ is a basis of $SM_{d_j}(D_{j})$, let $e_{B^{j}_{i,d}}$ denote the $d$-th component of $\e_{B^{j}_{i}}$.  Then
$e_{B^{j}_{1,d}}=e_{B^{j}_{2,d}}=\cdots =e_{B^{j}_{s_{j},d}}=1$, so the $d$-th component of $\p_j$ is $p_{j,d}=t(\lambda_{1}+\lambda_{2}+\cdots+\lambda_{s_{j}})=t$. If $d>\max(M(D_j))$, then there is no box in row $d$ of column $D_{j}$, so $e_{B^{j}_{1,d}}=e_{B^{j}_{2,d}}=\cdots =e_{B^{j}_{s_{j},d}}=0$, hence $p_{j,d}=0$.

If $d\in M(D_{j})$, then by the assumption $M(D_{i})\cap M(D_{j})=\emptyset$ for any $i\neq j$, the position in row $d$ of column $D_{i}$ is empty or an immovable box. Equivalently,
$e_{B^{i}_{1,d}}=e_{B^{i}_{2,d}}=\cdots =e_{B^{i}_{s_{i},d}}=0$ or $1$ for each $i\neq j$. By equation \eqref{pjj}, we have $p_{i,d}\in \{0,t\}$, then we find 
\begin{equation}\label{p_{j,d}}
    p_{j,d}=v_{d}-g_{d}\cdot t ,
\end{equation}
where $g_{d}=\# \{i\in [n]:i\neq j, \min(M(D_{i}))>d\}$.

 Since the components of each $\p_{j}$ before $\min(M(D_{j}))$ are all $t$, the components after $\max( M(D_{j}))$ are all $0$, and the components within  $M(D_{j})$ can be uniquely determined by equation \eqref{p_{j,d}}, we can conclude that each $\p_j$ is uniquely determined, and so is $\v$.

\noindent {\bf Case 2}. There exist at least two columns in  $D$ whose movable intervals intersect in exactly one element. Then, as described in the proof of Theorem \ref{thm: main.lattice-free}, we can partition all columns of $D$ into several blocks $\bigcup_{i=1}^{d} \mathcal{D}_{i}$.
Without loss of generality, we only analyze the columns in $\mathcal{D}=\{D_1,\ldots,D_k\}$ as shown in Figure \ref{fig:n}, and $\p_j\in tP(SM_{d_j}(D_{j}))$ for $j\in[k]$.

Since the movable intervals of  columns in $\mathcal{D}$ and  columns in $D \setminus \mathcal{D}$ do not intersect, we find that for any $i\notin [1, k]$ and $d\in[m_{1}, m_{k+1}]$, where $m_{j}=\min(M(D_{j}))$ for $1\leq j \leq k$, and $m_{k+1}=\max(M(D_{k}))$. Thus we have $p_{i,d}\in \{0,t\}$.

Now we analyze the $d$-th components of $\p_{1}$ for $d\in[m_{1}, m_{2}-1]$. Since $M(D_{1})\cap M(D_{2})=\{m_{2}\}$, it  is easy to see that the $d$-th components of $\p_{2},\p_{3},\ldots,\p_{k}$ are all equal to $t$, then we have 
\begin{equation}\label{ p_{1,d}}
          p_{1,d}=v_d-(k-1)\cdot t-\ell \cdot t,
\end{equation}
where $\ell=\# \{j\in [n]:\min(M(D_{j}))>m_{k+1}\}$.

Next we consider $p_{1,m_{2}}$.
    Suppose that the number of boxes in rows $m_{1},m_{1}+1,\ldots,m_{2}$ of $D_{1}$ is $k_{1}$. Then the sum of the $m_{1},m_{1}+1,\ldots,m_{2}$-th components of any basis vector of $SM_{d_1}(D_{1})$ is $k_{1}$.
 Since
    \begin{equation*}
        \p_{1}= \sum _{i=1}^{s_{1}}\lambda_{i}t e_{B^{1}_{i}},
    \end{equation*}
we have
 \[p_{1, m_{1}}+p_{1, m_{1}+1}+\cdots+p_{1, m_{2}}=\lambda_{1}tk_{1}+\lambda_{2}tk_{1}+\cdots+\lambda_{s_{1}}tk_{1}=tk_{1}.\]
Thus
\begin{equation}\label{p_1m2}
    p_{1, m_{2}}=tk_{1}-(p_{1, m_{1}}+p_{1, m_{1}+1}+\cdots+p_{1, m_{2}-1}).
\end{equation}
   Furthermore, since the boxes in rows $1,2,\ldots,m_{1}-1$ of $D_{1}$ are immovable and there are no boxes in rows $m_{2}+1,\ldots,n$, we find $p_{1,d}=t$ for $d\le m_1-1$ and $p_{1,d}=0$ for $d>m_2$. 

Similarly,
\begin{equation}\label{p_{2, m_{2}}}
    p_{2, m_{2}}=v_{m_{2}}-(k-2)\cdot t-\ell \cdot t-p_{1, m_{2}}.
\end{equation}
The analysis for the $m_{2}+1,\ldots,m_{3}$-th components of $\p_{2}$ is  the same with the $m_{1},\ldots,m_{2}$-th components of $\p_{1}$.

Observe the components of $\p_{1}$: (i)  components    before $\min(M(D_{1}))$ are all $t$, (ii) components after $\max( M(D_{1}))$ are all $0$, (iii) $p_{1, d}$ is uniquely determined by equation \eqref{ p_{1,d}} for  $d\in [m_{1}, m_{2}-1]$, and (iv) $p_{1, m_{2}}$ is uniquely determined by equation \eqref{p_1m2}, we can conclude that $\p_1$ is uniquely determined. Similarly,  each $\p_{i}$ for $i\in [k]$ is uniquely determined. Therefore,    any $ \v\in t\mathcal{S}_{D}\cap \mathbb{Z}^{n} $ in equation \eqref{eq5} is uniquely determined.
\end{proof}

\section{Applications}

In this section, we specialize the results for general diagrams in Section 3 and Section 4 to the Rothe diagram of permutations and the skyline diagram of compositions. Furthermore, we apply our lattice-free criteria to prove several recent conjectures regarding the support of Grothendieck polynomials.

\subsection{Newton polytopes of Schubert polynomials}

Recall that the Rothe diagram $D(w)$ of $w\in S_n$ is the set of boxes in the $n\times n$ grid which are not crossed by the $n$ hooks with corner at $(i,w(i))$ for $i\in[n]$.

\begin{lem}[\cite{FMD1}]\label{lem2}
Let $w \in S_n$ and $\sigma\in S_m$, and $w$ contains the pattern $\sigma$. Choose a realization $j_1 < j_2 < \cdots < j_m$ of $\sigma$ in $w$. Then $D(\sigma)$ is obtained from $D(w)$ by deleting the rows $[n]\backslash\{j_1, \ldots, j_m\}$ and the columns $[n]\backslash\{w(j_1), \ldots, w(j_m)\}$, and reindexing the remaining rows and columns by $[m]$, preserving their order.
\end{lem}

Corollary \ref{cor: cor 1.2} follows from Theorem \ref{thm: the main thm} and Proposition \ref{Schubert} below.

\begin{prop}\label{Schubert}
For a permutation $w\in S_n$, the following  conditions are equivalent:
\begin{enumerate}
  \item [$(1)$] The movable intervals of different columns in $D(w)$ are pairwise disjoint;
  \item [$(2)$] $w$ avoids the three permutation patterns 1423, 1432, 13254;
  \item [$(3)$] For any hook in $D(w)$, there is at most one column to its lower right containing boxes of $D(w)$.
\end{enumerate}
\end{prop}

\begin{proof}
  If there exists a hook whose lower right side has two columns containing boxes of $D(w)$, then the movable intervals of these two columns must intersect in at least two elements, so $(1) \Rightarrow (3)$ holds. Conversely, if there exist two columns $D_i, D_j(i<j)$ in the Rothe diagram $D(w)$ such that $|M(D_i)\cap M(D_j)| \geq 1$. Assume $M(D_i)=[c_1, c_2]$, we know that $D(w)$ has no box at $(c_1,i)$ but has a box at $(c_2,i)$. Therefore, a horizontal ray is present at $(c_1, i)$, and there must be no box at $(c_1, j)$. Furthermore, since $|M(D_i)\cap M(D_j)| \geq 1$, there exists $c_3 \in [c_1 +1, n]$ such that there is a box at $(c_3,j)$. Therefore, the hook with corner at $(c_1,w(c_1))$ has at least two columns to its lower right containing boxes of $D(w)$.
  Hence $(3) \Rightarrow (1)$ holds. So conditions (1) and (3) are equivalent.

Since in $D(1423), D(1432), D(13254)$ there exists a hook whose lower right side has two columns of boxes. It follows from Lemma \ref{lem2} that if $w$ contains one of the permutation patterns 1423, 1432, 13254, then there exists a hook in $D(w)$ whose lower right side has two columns of boxes, i.e., $(3) \Rightarrow (2)$ holds.

Finally, we aim to prove $(2)\Rightarrow (3)$, which is equivalent to proving its contrapositive: if there is a hook $L$ with corner at $(i,w(i))$, whose lower right side has at least two columns containing boxes in $D(w)$, then $w$ must contain one of the permutation patterns $1423, 1432, 13254$. 

Let $A=(i_1,a)$ be the leftmost and topmost box in the lower right region of the hook $L$. Then $i_1>i$ and $a>w(i)$. Let $B$ be the topmost box in another column under the hook $L$.
There are two cases for the position of $B$.

\noindent{\bf Case 1.} $B$ is in the same row of $A$.
Suppose that $B=(i_1,b)$ with $b>a$. Then there is a hook with corner in row $i_1$ and to the right of $B$. And there are two hooks with corners 
in the same column of $A$ and $B$ with relative positions illustrated in Figure \ref{fig:comparison}. It is easy to see that $w$ contains one of the patterns $1432$ or $1423$.
\begin{figure}[htbp]
  \centering
  \begin{subfigure}[b]{0.45\textwidth}
    \centering
    \begin{tikzpicture}[scale=0.7]
      \draw[line width=1pt](0mm, 0mm)--(0mm,44mm)--(56mm, 44mm);    
      \draw[fill=gray](8mm, 28mm)--(16mm, 28mm)--(16mm, 36mm)--(8mm, 36mm);
      
      \draw[fill=gray](24mm, 28mm)--(32mm, 28mm)--(32mm, 36mm)--(24mm, 36mm);
     
      \draw[line width=1pt] (36mm, 0mm)--(36mm, 32mm)--(56mm, 32mm);
      \draw[line width=1pt] (12mm, 0mm)--(12mm, 24mm)--(56mm, 24mm);
      \draw[line width=1pt] (28mm, 0mm)--(28mm, 16mm)--(56mm, 16mm);
      \node[left] at (0mm, 44mm){$i$};
      \node[left] at (0mm, 32mm){$i_1$};
   
      \node[above] at (12mm, 44mm){$a$};
      \node[above] at (28mm, 44mm){$b$};
      \end{tikzpicture}
  \end{subfigure}
    \hspace{0.1cm} 
  \begin{subfigure}[b]{0.45\textwidth}
    \centering
    \begin{tikzpicture}[scale=0.7]
      \draw[line width=1pt](0mm, 0mm)--(0mm,44mm)--(56mm, 44mm);    
      \draw[fill=gray](8mm, 28mm)--(16mm, 28mm)--(16mm, 36mm)--(8mm, 36mm);
      
      \draw[fill=gray](24mm, 28mm)--(32mm, 28mm)--(32mm, 36mm)--(24mm, 36mm);
      \draw[line width=1pt] (40mm, 0mm)--(40mm, 32mm)--(56mm, 32mm);
      \draw[line width=1pt] (12mm, 0mm)--(12mm, 14mm)--(56mm, 14mm);
      \draw[line width=1pt] (28mm, 0mm)--(28mm, 24mm)--(56mm, 24mm);
      \node[left] at (0mm, 44mm){$i$};
      \node[left] at (0mm, 32mm){$i_1$};
      \node[above] at (12mm, 44mm){$a$};
      \node[above] at (28mm, 44mm){$b$};
    \end{tikzpicture}
  \end{subfigure}
  \caption{The illustration of Case 1}
  \label{fig:comparison}
\end{figure}
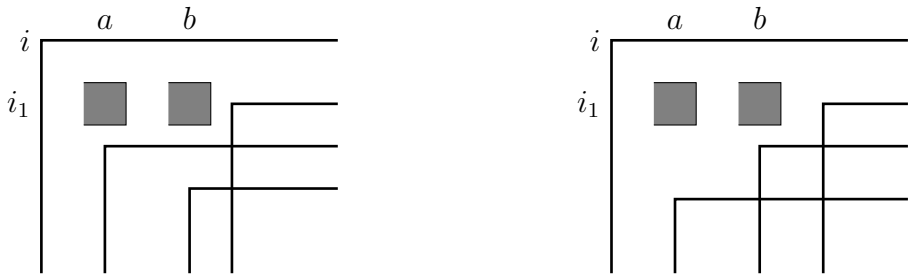

\noindent{\bf Case 2.} $B$ is not in the same row of $A$. 
Then $B$ must locate in a lower row of $A$. Let $B=(i_2,b)$ with $i_2>i_1$ and $b>a$. Then there is a hook with corner at $(i_{1},w(i_1))$ with $b>w(i_{1})>a$. Now consider the position at $(i_{2},a)$.
If there is a box at $(i_2,a)$, then delete row $i_1$ and the column $w(i_1)$, the discussion returns to Case 1, see the left figure in Figure \ref{fig:comparison2}. By Lemma \ref{lem2}, $w$ contains one of the patterns $1432$ or $1423$.

If there is no box at $(i_2,a)$, then this position must be crossed by a vertical ray. This implies that there is a hook with corner at $(i_{3},a)$ with $i_{3}<i_{2}$. Moreover, due to the existence of $B$, there is a hook with corner in row $i_2$ to the right of $B$ and a hook with corner under $B$ with relative positions illustrated in the right figure of Figure \ref{fig:comparison2}. It is easy to see that $w$ contains the pattern $13254$.
\end{proof}

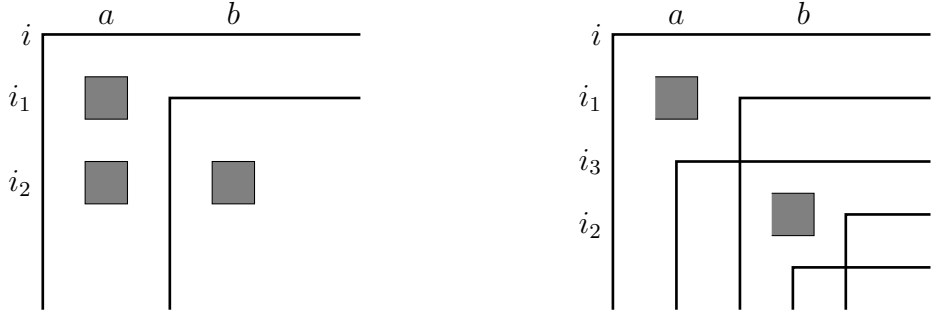
\begin{figure}[htbp]
  \centering
  \begin{subfigure}[b]{0.45\textwidth}
    \centering
    \begin{tikzpicture}[scale=0.7]
     \draw [line width=1pt] (0mm, 0mm) -- (0mm, 52mm) -- (60mm, 52mm);
\draw [line width=1pt] (24mm, 0mm) -- (24mm, 40mm) -- (60mm, 40mm);
\draw[fill=gray] (8mm, 36mm) -- (16mm, 36mm) -- (16mm, 44mm) -- (8mm, 44mm) -- cycle;
\draw[fill=gray] (8mm, 20mm) -- (16mm, 20mm) -- (16mm, 28mm) -- (8mm, 28mm) -- cycle;
\draw[fill=gray] (32mm, 20mm) -- (40mm, 20mm) -- (40mm, 28mm) -- (32mm, 28mm) -- cycle;
\node[left] at (0mm, 52mm) {$i$};
\node[left] at (0mm, 40mm) {$i_1$};
\node[left] at (0mm, 24mm) {$i_2$};

\node[above] at (12mm, 52mm) {$a$};
\node[above] at (36mm, 52mm) {$b$};
      \end{tikzpicture}
  \end{subfigure}
    \hspace{0.1cm} 
 \begin{subfigure}[b]{0.45\textwidth}
    \centering
    \begin{tikzpicture}[scale=0.7]
\draw[line width=1pt](0mm, 0mm)--( 0mm,52mm)--(60mm, 52mm);    
\draw[fill=gray](8mm, 36mm)--(16mm, 36mm)--(16mm, 44mm)--(8mm, 44mm);

\draw [line width=1pt] (24mm, 0mm)--(24mm, 40mm)--(60mm, 40mm);
\draw[line width=1pt] (12mm, 0mm)--(12mm, 28mm)--(60mm, 28mm);
\draw[fill=gray](30mm, 14mm)--(38mm, 14mm)--(38mm, 22mm)--(30mm,22mm);
\draw [line width=1pt] (44mm, 0mm)--(44mm, 18mm)--(60mm, 18mm);
\draw [line width=1pt] (34mm, 0mm)--(34mm, 8mm)--(60mm, 8mm);
\node[left ] at (0mm, 52mm){$i$};
\node[left] at (0mm, 40mm){$i_1$};
\node[left] at (0mm, 28mm){$i_3$};   
\node[left] at (0mm, 16mm){$i_2$};
\node[above ] at (12mm, 52mm){$a$};
\node[above] at (36mm, 52mm){$b$};
\end{tikzpicture}
  \end{subfigure}
  \caption{The illustration of Case 2}
  \label{fig:comparison2}
\end{figure}

\subsection{Newton polytopes of Grothendieck polynomials}

Recently, M\'esz\'aros, Setiabrata, and St.\,Dizier \cite{MSSD} conducted a systematic study on the support $\supp(\mathfrak{G}_w)$ of Grothendieck polynomials. Equipping $\supp(\mathfrak{G}_w) \subset \mathbb{Z}^n$ with the componentwise comparison order $\alpha \le \beta$, they proposed several conjectures (Conjectures 1.1--1.6 in \cite{MSSD}) detailing the combinatorial structure and saturatedness of the Newton polytopes. 

In this subsection, we show that the Newton polytope $\Newton(\mathfrak{G}_w)$ becomes lattice-free if and only if $w$ avoids the permutation patterns $1432, 1423, 13254$. Furthermore, the geometric consequences of this avoidance allow us to  verify all the aforementioned conjectures for this class of permutations. 

Before presenting the main theorems, we need a   lemma revealing that the geometric bounding box of $\Newton(\mathfrak{G}_w)$ is strictly constrained by that of $\mathcal{S}_w=\Newton(\S_w)$. Let $d_i(f)$ denote the maximum degree of the variable $x_i$ appearing in a polynomial $f$.

\begin{lem}\label{lem:max_degree_grothendieck}
For any permutation $w \in S_n$ and any $i \in [n]$, the maximum degree of $x_i$ in the Grothendieck polynomial $\mathfrak{G}_w(x)$ is exactly equal to its maximum degree in the Schubert polynomial $\S_w(x)$, i.e.,
\[
d_i(\mathfrak{G}_w) = d_i(\S_w).
\]
\end{lem}

\begin{proof}
Since $\S_w(x)$ is the lowest-degree homogeneous component of $\mathfrak{G}_w(x)$, we naturally have $d_i(\S_w) \le d_i(\mathfrak{G}_w)$.
By \cite[Theorem 1.2]{MSSD22}, for any $\alpha \in \supp(\mathfrak{G}_w)$, we have $\alpha \le \mathrm{wt}(\overline{D(w)})$, where $\overline{D(w)}$ is the upper closure of the Rothe diagram $D(w)$. Hence, $d_i(\mathfrak{G}_w) \le \mathrm{wt}(\overline{D(w)})_i$.
On the other hand, by Theorem \ref{S_D and minkovski sum}, $\Newton(\S_w) = \sum_{j=1}^n P(SM_{d_j}(D_j))$. The maximum $i$-th coordinate in the Schubert matroid polytope $P(SM_{d_j}(D_j))$ is 1 if and only if $i$ is an element in a basis of $SM_{d_j}(D_j)$. Since the bases are subsets $T$ satisfying $T \preceq D_j$ in Gale order, $i$ can be included in a basis if and only if $i \le \max(D_j)$. This inequality $i \le \max(D_j)$ is precisely the defining condition for a box to appear in the $j$-th column of the upper closure $\overline{D(w)}$. Therefore, the maximum $i$-th coordinate in the Minkowski sum $\Newton(\S_w)$ is exactly the number of columns $j$ such that $i \le \max(D_j)$, which is exactly  $\mathrm{wt}(\overline{D(w)})_i$. This implies $d_i(\S_w) = \mathrm{wt}(\overline{D(w)})_i$, and consequently $d_i(\mathfrak{G}_w) \le d_i(\S_w)$.
\end{proof}

\begin{thm}\label{thm:Grothendieck_lattice_free}
For a permutation $w \in S_n$, the Newton polytope $\Newton(\mathfrak{G}_w)$ is lattice-free if and only if $w$ avoids the patterns $1432, 1423, 13254$.
\end{thm}
\begin{proof}
We first prove the necessity. As $\S_w(x)$ is the lowest-degree homogeneous component of $\mathfrak{G}_w(x)$, the Newton polytope $\Newton(\S_w)$ is a face of $\Newton(\mathfrak{G}_w)$. If $\Newton(\mathfrak{G}_w)$ is an integral lattice-free polytope, its face $\Newton(\S_w)$ must also be lattice-free. By Corollary \ref{cor: cor 1.2}, $w$ avoids $1432, 1423, 13254$.

For the sufficiency, suppose $w$ avoids $1432, 1423, 13254$. By Proposition \ref{Schubert}, the movable intervals $M(D_j)$ of the columns in the Rothe diagram $D(w)$ are pairwise disjoint. This implies that for any given row index $i \in [n]$, the variable $x_i$ can have a non-constant degree in at most one column's Schubert matroid polytope $P(SM_{d_j}(D_j))$. In all other columns $k \neq j$, the $i$-th coordinate of any point in $P(SM_{d_k}(D_k))$ is a fixed constant $c_{i,k} \in \{0,1\}$. 
Let $C_i = \sum_{k \neq j} c_{i,k}$. Since the degree contribution from the active column $j$ is bounded between 0 and 1, the width (maximum minus minimum value) of the $i$-th coordinate across the entire Minkowski sum $\Newton(\S_w)$ is at most 1. Hence, $\Newton(\S_w)$ is tightly enclosed in the translated unit hypercube $\mathcal{Q} = \prod_{i=1}^n[C_i, C_i+1]$.

By Lemma \ref{lem:max_degree_grothendieck}, the maximum degree of each $x_i$ in $\mathfrak{G}_w(x)$ is exactly equal to its maximum degree in $\S_w(x)$. By Fomin and Kirillov \cite{FK}, each non-reduced pipe dream contains a reduced pipe dream by replacing  the extra crossings  of two pipes with bumping elbows, thus the minimum degree of $x_i$ in $\mathfrak{G}_w(x)$ is also bounded by the minimum degrees found in $\S_w(x)$. 
Therefore, $\Newton(\mathfrak{G}_w)$ is geometrically bounded by the exact same translated unit hypercube $\mathcal{Q}$. Since any integral polytope  contained within a unit hypercube is  a $0/1$-polytope (up to translation), such polytopes naturally contain no non-vertex lattice points. Hence, $\Newton(\mathfrak{G}_w)$ is lattice-free.
\end{proof}

 Recall a well known property in matroid theory (see, e.g., Schrijver \cite[Section 39.3]{A.S}): a subset $S$ is a spanning set of a matroid $M$ on the ground set $E$ if and only if its complement $E \setminus S$ is an independent set in the dual matroid $M^*$. Consequently, the spanning set polytope $P_{\mathrm{sp}}(M)$ is an affine transformation of the independence polytope $P_{\mathrm{ind}}(M^*)$, specifically given by 
 \begin{equation}\label{eq:spanind}
     P_{\mathrm{sp}}(M) = \mathbf{1} - P_{\mathrm{ind}}(M^*).
 \end{equation}
Since $P_{\mathrm{ind}}(M^*)$ is a standard integral polymatroid, its Minkowski sum possesses the integer decomposition property (IDP)  by Schrijver \cite[Corollary 46.2c]{A.S}.
Therefore, the integer decomposition property  also holds for the Minkowski sum of spanning set polytopes.

\begin{lem}\label{spandingposi}
    The spanning set polytope of a Schubert matroid is Ehrhart positive. 
\end{lem}
\begin{proof}
 By \eqref{eq:spanind}, we have $i(P_{\mathrm{sp}}(M),t)=i(P_{\mathrm{ind}}(M^*),t).$
Since the dual of a Schubert matroid is isomorphic, after reversing the ground-set order, to a Schubert matroid, it suffices to show that $P_{\mathrm{ind}}(M)$ is Ehrhart positive for every Schubert matroid $M$. 

Suppose $M$ has rank $r$. Define $Q_k(M)=\{x\in P_{\mathrm{ind}}(M): k-1<\sum_{i\in E}x_i\leq k\}$, then we have $$tP_{\mathrm{ind}}(M)\cap \mathbb{Z}^n=\{0\}\sqcup\bigsqcup_{k=1}^{r}(tQ_k(M)\cap \mathbb{Z}^n).$$
Let \(E_k(t)
=
\#\left\{
x\in tP_{\mathrm{ind}}(M)\cap\mathbb Z^n:
t(k-1)<\sum_{i\in E}x_i\le tk
\right\}.\)
Then $$i(P_{\mathrm{ind}}(M),t)=1+\sum_{k=1}^{r}E_k(t).$$
We aim to show that each $E_k(t)$ is a polynomial with nonnegative coefficients when $M$ is a Schubert matroid. 

Let \(\operatorname{Tr}_j(M)\) denote the rank-\(j\) truncation of the matroid \(M\), whose independent sets are precisely the independent sets of \(M\) having cardinality at most \(j\). Let \(M_1\) and \(M_2\) be matroids on disjoint ground sets \(E_1\) and \(E_2\), respectively. The direct sum \(M_1\oplus M_2\) is the matroid on \(E_1\sqcup E_2\) whose independent sets are precisely the sets of the form \(I_1\sqcup I_2\), where \(I_i\) is independent in \(M_i\) for \(i=1,2\).

Consider $N_k=\operatorname{Tr}_k(U_{1,1}\oplus M)$ and $F_k=U_{1,1}\oplus\operatorname{Tr}_{k-1}(M).$ Then we have 
\begin{align*}
P(N_k)&=\left\{(x_e,x): x\in P_{\mathrm{ind}}(M), 0\leq x_e\leq 1, \sum_{i\in E}x_i+x_e=k \right\},\\
P(F_k)&=\left\{(x_e,x): x\in P_{\mathrm{ind}}(M), x_e=1, \sum_{i\in E}x_i=k-1 \right\}.
\end{align*}
As a result, 
\begin{equation}\label{eq:ehrqk}
    E_k(t)=i(P(N_k),t)-i(P(F_k),t).
\end{equation}

Now let $M=SM_n(S)$ with $S=\{s_1<\cdots< s_r\}$.  
Adding a new coloop $e$ at the first position, we obtain 
$$U_{1,1}\oplus SM_n(S)=SM_{n+1}(\{1,s_1+1,\ldots,s_r+1\}).$$
Note that a $j$-subset $A=\{a_1<\cdots<a_j\}$ with $j<r$ is an independent set in $SM_n(S)$ if and only if $a_i\leq s_{r-j+i}$ for $i\leq j$.
Taking the rank-$k$ truncation, we have
\begin{align*}
    N_k=SM_{n+1}(\{s_{r-k+1}+1,\ldots,s_r+1\}).
\end{align*}
Similarly,
$
F_k=SM_{n+1}(\{1,s_{r-k+2}+1,\ldots,s_r+1\}).
$

Since a Schubert matroid is a special lattice path matroid with trivial upper path; see for example, \cite[Corollary 4.4]{BM} or \cite[Theorem 2.8]{FS}.
It is easy to see that the lattice path determined by $\{1,s_{r-k+2}+1,\ldots,s_r+1\}$ is  contained in the lattice path determined by $\{s_{r-k+1}+1,\ldots,s_r+1\}$. By \cite[Theorem 4.1]{FMP}, we have 
$$i(P(F_k),t)\preceq i(P(N_k),t),$$
where "$\preceq$" stands for coefficient-wise inequality.
Combining with \eqref{eq:ehrqk}, we can conclude that $P_{\mathrm{ind}}(M)$ is Ehrhart positive for any Schubert matroid $M$. 
\end{proof}

\begin{proof}[Proof of Corollary \ref{corforGrothendicek}]
By Corollary \ref{cor: cor 1.2} and Theorem \ref{thm:Grothendieck_lattice_free}, we need only to show that for a permutation $w \in S_n$,  the Ehrhart polynomial of the Grothendieck Newton polytope factors as
\begin{equation}\label{eq:Grothendieck_Ehrhart_product}
i(\Newton(\mathfrak{G}_w), t) = \prod_{j=1}^n i(P_{\mathrm{sp}}(SM_{d_j}(D_j)), t)
\end{equation}
if and only if $\Newton(\mathfrak{G}_w)$ is lattice-free.
%\end{thm}

We first prove the necessity. Suppose to the contrary that $\Newton(\mathfrak{G}_w)$ is not lattice-free. By Theorem \ref{thm:Grothendieck_lattice_free} and Proposition \ref{Schubert}, $w$ contains at least one of the patterns $1432, 1423, 13254$, which implies that there exist two columns, say $D_1$ and $D_2$, in the Rothe diagram $D(w)$ whose movable intervals intersect in at least two elements.

Let $Q = \sum_{j=1}^n P_{\mathrm{sp}}(SM_{d_j}(D_j))$. By \cite[Theorem 3.6]{MSSD}, we have $\Newton(\mathfrak{G}_w) \subseteq Q$. Thus, for any positive integer $t$, we naturally have $t\Newton(\mathfrak{G}_w) \subseteq tQ$.
Recall that every basis of a matroid is inherently a spanning set. Thus $P(SM_{d_j}(D_j)) \subseteq P_{\mathrm{sp}}(SM_{d_j}(D_j))$ for all $j \in [n]$. Since $|M(D_1) \cap M(D_2)| \ge 2$, we can construct the exact same vertices $\v_1, \v_2$ and their midpoint $\overline{\v}$ as defined in the proof of Theorem \ref{thm: main.lattice-free}. The point $\overline{\v}$ resides in $Q$ and can be decomposed into the sum of points from the component spanning set polytopes in at least two different ways.
By simply scaling these integer components by any integer $t \ge 1$, the dilated point $t\overline{\v} \in tQ \cap \mathbb{Z}^n$ correspondingly admits at least two distinct valid decompositions.
Since the integer decomposition property (IDP)  holds for spanning set polytopes, every lattice point in $tQ \cap \mathbb{Z}^n$ can be expressed as a sum of lattice points from each summand $tP_{\mathrm{sp}}(SM_{d_j}(D_j)) \cap \mathbb{Z}^n$. Therefore,
\begin{align*}
|tQ \cap \mathbb{Z}^n| < \prod_{j=1}^n |tP_{\mathrm{sp}}(SM_{d_j}(D_j)) \cap \mathbb{Z}^n|.
\end{align*}
Since $t\Newton(\mathfrak{G}_w) \subseteq tQ$, it follows that
\[
i(\Newton(\mathfrak{G}_w), t) = |t\Newton(\mathfrak{G}_w) \cap \mathbb{Z}^n| \le |tQ \cap \mathbb{Z}^n| < \prod_{j=1}^n i(P_{\mathrm{sp}}(SM_{d_j}(D_j)), t),
\]
which  contradicts the assumption $i(\Newton(\mathfrak{G}_w), t) = \prod_{j=1}^n i(P_{\mathrm{sp}}(SM_{d_j}(D_j)), t)$.

Now we prove the sufficiency. Suppose that $\Newton(\mathfrak{G}_w)$ is lattice-free. By Theorem \ref{thm:Grothendieck_lattice_free}, $w$ avoids $1432$, $1423$, and $13254$. Then by Proposition \ref{Schubert}, the movable intervals of different columns in $D(w)$ are pairwise disjoint. 
This geometric decoupling implies that the variables associated with different columns are independent. As shown by Fan and Guo \cite{FG2}, for $1432$-avoiding permutations, the monomials of $\mathfrak{G}_w(x)$ are generated by set-fillings of $D(w)$ such that the rows are weakly decreasing and columns are strictly increasing,  with the numbers filling in row $i$ less than or equal to $i$. Due to the disjointness of the movable intervals, we can ignore the condition that rows are weakly decreasing in this set-filling process for each column $D_j$. 

Crucially, for a given column, such an increasing set-filling precisely generates a spanning set of the associated Schubert matroid $SM_{d_j}(D_j)$. Conversely, every spanning set admits at least one such valid set-filling. Although a single spanning set $A$ might correspond to multiple valid set-fillings, all such fillings generate the same monomial with sign $(-1)^{|A|-|D_j|}$, meaning there is no cancellation. This guarantees that the support of the polynomial exactly matches the integer points of the Minkowski sum of the spanning set polytopes. Consequently, 
\begin{equation}\label{eq:Newton_Minkowski_Spanning}
\Newton(\mathfrak{G}_w) = \sum_{j=1}^n P_{\mathrm{sp}}(SM_{d_j}(D_j)).
\end{equation}

Moreover, because the movable intervals are pairwise disjoint, except the  coordinates corresponding to the upper left fixed partition of $D(w)$, the coordinates of the component polytopes $P_{\mathrm{sp}}(SM_{d_j}(D_j))$ reside in mutually orthogonal affine subspaces after integral translations. Thus, this Minkowski sum degenerates into a Cartesian product. Since the Ehrhart polynomial of a Cartesian product of polytopes is simply the product of the Ehrhart polynomials of its factors, we immediately obtain \eqref{eq:Grothendieck_Ehrhart_product}. 
By Lemma \ref{spandingposi}, $\Newton(\mathfrak{G}_w)$ is Ehrhart positive.
This completes the proof.
\end{proof}

\begin{remark}
It is worth noting that the $\{1432, 1423, 13254\}$-avoiding permutations form a strict subclass of permutations whose Schubert polynomials are zero-one (which require avoiding 12 longer patterns \cite{FMD1}). Consequently, the saturatedness and generalized polymatroid property for our class (Conjecture 1.4 of \cite{MSSD}) is theoretically covered by Castillo et al. \cite[Theorem B]{CCM}, who proved it for all zero-one Schubert polynomials using multiplicity-free varieties. 

However, we still include this property in the following theorem, since for lattice-free Schubitopes, we can provide a much simpler and  elementary  proof that resolves Conjectures 1.1--1.4 and 1.6 of \cite{MSSD} simultaneously.
\end{remark}

\begin{thm}\label{thm:conjectures_proof}
If a permutation $w \in S_n$ avoids the patterns $1432, 1423, 13254$, then Conjectures 1.1--1.4 and 1.6 from \cite{MSSD} all hold for $\mathfrak{G}_w(x)$. Specifically:
\begin{enumerate}
    \item[$(1)$] If $\alpha \in \supp(\mathfrak{G}_w)$ and $|\alpha| < \deg \mathfrak{G}_w$, there exists $\beta \in \supp(\mathfrak{G}_w)$ with $\alpha < \beta$ and $|\beta| = |\alpha| + 1$.
    \item[$(2)$] $\supp(\mathfrak{G}_w)$ is closed under taking intervals in componentwise comparison.
    \item[$(3)$] $\G_w(x)$ has SNP, and $\Newton(\mathfrak{G}_w)$ is a generalized polymatroid.
    \item[$(4)$] The sum of the coefficients under any top-degree component of $\mathfrak{G}_w(x)$ is 1.
\end{enumerate}
\end{thm}

\begin{proof}
 The maximum degree $\deg \mathfrak{G}_w$ corresponds to the unique maximal spanning set where all available boxes in each movable interval $M(D_j)$ are completely filled. If $\alpha \in \supp(\mathfrak{G}_w)$ satisfies $|\alpha| < \deg \mathfrak{G}_w$, then in at least one column $j$, the corresponding spanning set $A_j$ is not full. We can add exactly one element $x \in M(D_j) \setminus A_j$ to obtain $B_j = A_j \cup \{x\}$, which remains a spanning set. The updated exponent vector $\beta$ satisfies $\alpha < \beta$ and $|\beta| = |\alpha| + 1$, confirming Conjectures 1.1 and 1.2 of \cite{MSSD}.

 Since Conjecture 1.4 is a strengthening of Conjecture 1.3, we need only to give a proof of Conjecture 1.4. 
Since each spanning set polytope is a generalized polymatroid \cite[Proposition 2.16]{MSSD}, and generalized polymatroids are closed under Minkowski sum \cite[Proposition 2.15]{MSSD}, by \eqref{eq:Newton_Minkowski_Spanning}, $\Newton(\mathfrak{G}_w)$ is a generalized polymatroid. This confirms Conjecture 1.4 of \cite{MSSD}.

 Because the columns are fully decoupled, there is no competition for variable degrees. 
 The top-degree homogeneous component corresponds exclusively to the case where the chosen spanning set for each column $j$ is the maximal possible subset, namely the entire movable interval $M(D_j)$. Since the set-filling must be strictly increasing along the columns, each element of $M(D_j)$ can appear at most once in column $j$. Therefore, to maximize the degree, every element of $M(D_j)$ must appear exactly once in column $j$, regardless of how these elements are distributed among the boxes. Thus $\supp(\mathfrak{G}_w)$ has a unique maximum element $\beta_{\max}$, and the top-degree component consists of a single term. Consequently, the summation $\sum_{\alpha \le \beta_{\max}} C_{w\alpha}$ requested in Conjecture 1.6 equals the sum of all coefficients across the entire polynomial $\mathfrak{G}_w(x)$. By     \cite[Proposition 4.4]{MSSD}, the principal specialization evaluates to $\mathfrak{G}_w(1,\dots,1) = 1$. Thus, the sum of all coefficients is exactly 1, confirming Conjecture 1.6 of \cite{MSSD}.
\end{proof}

\subsection{Newton polytopes of key polynomials}

Let $\alpha=(\alpha_{1},\ldots, \alpha_{n})\in \mathbb{Z}_{\ge 0}^{n}$. For $1\le i < j\le n$, let $t_{i,j}(\alpha)$ denote the composition obtained from $\alpha$ by interchanging $\alpha_{i}$ and $\alpha_{j}$, and let
$$m_{i,j}(\alpha)=\alpha+\e_{i}-\e_{j}.$$ 
If a vector $\beta\in \mathbb{Z}_{\ge 0}^{n}$ can be obtained from $\alpha$ by applying a sequence of $t_{i,j}$ (when $\alpha_{i}<\alpha_{j}$) and $m_{i,j}$ (when $\alpha_{i}\le\alpha_{j}-2$), then we denote $\beta\le_{k}\alpha.$ In particular, if $\beta$ is obtained from $\alpha$ by applying only a sequence of $t_{i,j}$ (when $\alpha_{i}<\alpha_{j}$), we denote $\beta\le \alpha.$

\begin{prop}[\cite{NG1,NG2}]\label{lattice point set}
Let $\alpha\in \mathbb{Z}_{\ge 0}^{n}$ be a composition. Then the set of vertices of $\mathcal{S}_\alpha=\Newton(\kappa_{\alpha})$ is
$\{\beta:\beta\le \alpha\}$. The set of lattice points of $\mathcal{S}_\alpha$ is
$\{\beta:\beta\le_k \alpha\}$. 
\end{prop}
 
Corollary \ref{cor: cor 1.3} holds by Theorem \ref{thm: the main thm} and the following Proposition.

\begin{prop}
For a composition $\alpha\in \mathbb{Z}_{\ge 0}^{n}$,
$\mathcal{S}_\alpha$ is lattice-free if and only if $\alpha$ avoids the composition pattern $(0,2)$, i.e., there do not exist $i<j$ such that $\alpha_{j}-\alpha_{i}\ge 2.$
\end{prop}
\begin{proof}
 To prove the necessity, suppose to the contrary that  $\alpha$ contains the composition pattern $(0,2)$, i.e., if there exist $i<j$ such that $\alpha_{j}-\alpha_{i}\ge 2$. Then the skyline diagram $D(\alpha)$ must contain two columns whose movable intervals intersect in at least two elements. By Theorem \ref{thm: the main thm}, $\mathcal{S}_\alpha$ must contain non-vertex lattice points, which is a contradiction.

On the other hand, if $\alpha$ avoids the composition pattern $(0,2)$, i.e., there do not exist $i<j$ such that $\alpha_{j}-\alpha_{i}\ge 2$, then by Proposition \ref{lattice point set}, any lattice point $\beta$ in $\mathcal{S}_\alpha$ can only be obtained from $\alpha$ by applying a sequence of  $t_{i,j}$ (when $\alpha_{i}<\alpha_{j}$). Consequently, all lattice points $\beta$ in $\mathcal{S}_\alpha$ are vertices, i.e., $\mathcal{S}_\alpha$ is lattice-free. 
\end{proof}

\end{document}